\documentclass[a4paper, 11pt]{amsart}

\title[Categorical blow-up formula]{Categorical blow-up formula for Hilbert schemes of points}
\author{Naoki Koseki}
\date{}

\address{The University of Liverpool, Mathematical Sciences Building, Liverpool, L69 7ZL, UK.}
\email{koseki@liverpool.ac.uk}

\makeatletter
 
  \@addtoreset{equation}{section}
 \makeatother

\usepackage{amsmath, amssymb, amsthm, amscd, comment, mathtools, color}
\usepackage{todonotes}
\usepackage[frame,cmtip,curve,arrow,matrix,line,graph]{xy}
\usepackage[mathscr]{euscript}
\usepackage{mathrsfs}
\usepackage{tikz}
\usetikzlibrary{intersections, calc}

\theoremstyle{plain}
\newtheorem{thm}{Theorem}[section]
\newtheorem{prop}[thm]{Proposition}
\newtheorem{def-prop}[thm]{Definition-Proposition}
\newtheorem{lem}[thm]{Lemma}
\newtheorem{cor}[thm]{Corollary}
\newtheorem*{thm*}{Theorem}

\theoremstyle{definition}
\newtheorem{defin}[thm]{Definition}

\newtheorem*{NaC}{Notation and Convention}
\newtheorem*{ACK}{Acknowledgement}

\theoremstyle{remark}
\newtheorem{rmk}[thm]{Remark}

\newtheorem{assump}[thm]{Assumption}

\DeclareMathOperator{\ch}{ch}

\DeclareMathOperator{\rk}{rk}

\DeclareMathOperator{\Spec}{Spec}
\DeclareMathOperator{\id}{id}

\newcommand{\bP}{\mathbb{P}}
\newcommand{\bC}{\mathbb{C}}

\newcommand{\bQ}{\mathbb{Q}}
\newcommand{\bZ}{\mathbb{Z}}

\newcommand{\mcC}{\mathcal{C}}

\newcommand{\mcE}{\mathcal{E}}
\newcommand{\mcF}{\mathcal{F}}
\newcommand{\mcG}{\mathcal{G}}
\newcommand{\mcH}{\mathcal{H}}

\newcommand{\mcK}{\mathcal{K}}

\newcommand{\mcO}{\mathcal{O}}

\newcommand{\mcV}{\mathcal{V}}

\DeclareMathOperator{\Hom}{Hom}
\DeclareMathOperator{\Tot}{Tot}
\DeclareMathOperator{\Coh}{Coh}

\DeclareMathOperator{\ext}{ext}
\DeclareMathOperator{\Ext}{Ext}

\DeclareMathOperator{\Hilb}{Hilb}
\DeclareMathOperator{\Sym}{Sym}

\DeclareMathOperator{\Quot}{Quot}

\DeclareMathOperator{\tr}{tr}

\newcommand{\hatS}{\hat{S}}

\begin{document}
\maketitle
\begin{abstract}
Let $S$ be a smooth projective surface, 
and $\hatS$ be its blow-up at a point. 
In this paper, we study the derived category of 
the Hilbert scheme of points on the blow-up $\hatS$. 
We obtain a semi-orthogonal decomposition consisting of 
the derived categories of the Hilbert schemes on the original surface $S$, 
which recovers the blow-up formula for the Euler characteristics obtained by 
G{\"o}ttsche and Nakajima--Yoshioka. 

The proof uses the Quot formula, 
which was conjectured by Jiang and recently proved by Toda. 
\end{abstract}

\setcounter{tocdepth}{1}
\tableofcontents

%=======================================================================
\section{Introduction}
\subsection{Main Results}
Hilbert schemes of points on a smooth projective surface $S$ have been investigated 
in various contexts, especially in the geometric representation theory (see, e.g., \cite{nak99}). 
A starting point of the development would be the following G{\"o}ttsche's formula (\cite{got90}):  
\begin{equation} \label{eq:got}
Z_S(q) \coloneqq \sum_{n\geq 0} e(\Hilb^n(S))q^n
=\left(\prod_{d=1}^\infty \frac{1}{1-q^d} \right)^{e(S)}. 
\end{equation}

The subject of this paper is to investigate 
the change of the invariants 
of the Hilbert schemes under the blow-up of a surface. 
Let $\hatS \to S$ be the blow-up of $S$ at a point. 
Then the formula (\ref{eq:got}) immediately implies 
\begin{equation} \label{eq:blformula}
Z_{\hatS}(q)=\left(\prod_{d=1}^\infty \frac{1}{1-q^d} \right) \cdot Z_S(q), 
\end{equation}
which we call {\it the blow-up formula}. 
Nakajima--Yoshioka \cite{ny11} gave a geometric understanding of 
the blow-up formula (\ref{eq:blformula}). 
They constructed a sequence 
\begin{equation} \label{eq:introny}
\Hilb^n(S) \cong M^0(1, 0, -n) \dashleftarrow M^1(1, 0, -n) \dashleftarrow 
\cdots \dashleftarrow M^m(1, 0, -n) \cong \Hilb^n(\hatS)
\end{equation}
of biratinal maps, where $M^k(1, 0, -n)$ ($k \in \bZ_{\geq 0}$)
is the moduli space of certain coherent sheaves on $\hatS$ 
with Chern character $(1, 0, -n)$, 
and $m \gg 0$. 
Furthermore, they obtained the formula for the Euler characteristics of these moduli spaces: 
\begin{equation}\label{eq:introMk}
\sum_{n \geq 0}e(M^k(1, 0, -n))q^n
=\left(\prod_{d=1}^k \frac{1}{1-q^d} \right) \cdot Z_S(q), 
\end{equation}
which recovers the blow-up formula (\ref{eq:blformula}) 
by taking the limit $k \to +\infty$. 

In this paper, we obtain the following categorification of the formulas 
(\ref{eq:blformula}) and (\ref{eq:introMk}) 
in terms of semi-orthogonal decompositions (SODs) of derived categories: 
\begin{thm}[Special case of Theorem \ref{thm:sodfinal}] 
\label{thm:introhilb}
The following statements hold: 
\begin{enumerate}
\item For each $k, n \in \bZ_{\geq 0}$, we have an SOD 
\[
D^b(M^k(1, 0, -n))=\left\langle
A_k(j)\mbox{-copies of } D^b(\Hilb^{n-j}(S)) \colon j=0, \cdots n
\right\rangle, 
\]
where $A_k(j)$ is defined by the formula 
\[
\prod_{d=1}^{k}\frac{1}{(1-q^d)}=\sum_j A_k(j)q^j. 
\]

\item For each $n \in \bZ_{\geq 0}$, we have an SOD 
\[
D^b(\Hilb^n(\hatS))=\left\langle
p(j)\mbox{-copies of } D^b(\Hilb^{n-j}(S)) \colon j=0, \cdots n
\right\rangle, 
\]
where $p(j)$ is the partition function. 
\end{enumerate}
\end{thm}
Indeed, by taking the Euler characteristics of 
the Hochschild homology groups of the derived categories, 
we recover the formulas (\ref{eq:blformula}) and (\ref{eq:introMk}). 
Note that in the previous work \cite{kos21}, the author proved that 
the sequence (\ref{eq:introny}) gives steps of a minimal model program 
for the Hilbert scheme $\Hilb^n(\hatS)$ on the blow-up. 
Hence we obtain an interesting relationship among three different fields; 
numerical invariants, birational geometry, and derived categories. 

Note that Theorem \ref{thm:sodfinal} is much more general and works 
for moduli spaces of higher rank sheaves 
when the surface $S$ is del Pezzo, K3, or abelian. 

%=======================================================================
\subsection{Sketch of the proof}
The main tool for the proof of Theorem \ref{thm:introhilb} is the {\it Quot formula}, 
which was conjectured by Jiang \cite{jia21} and recently proved by Toda \cite{tod21}. 
See Theorem \ref{thm:quotformula} for the precise statement. 

The key observation is that we can interpret the moduli space 
$M^k(1, 0, -n)$ as 
a certain Quot scheme on the Hilbert scheme $\Hilb^n(S)$ on the original surface $S$ 
(see Corollary \ref{cor:quot}). 
Then we can apply the Quot formula to the moduli spaces $M^k(1, 0, -n)$ 
to relate its derived category with the derived categories of 
$M^{k-1}(1, 0, j-n)$. 
Inductively, we get the SODs in Theorem \ref{thm:introhilb}.

%=======================================================================
\subsection{Related works}
\begin{enumerate}
\item In the previous paper \cite{kos21}, the author constructed a fully faithful embedding 
$D^b(M^k(1, 0, -n)) \hookrightarrow D^b(M^{k+1}(1, 0, -n))$ 
for each $n \geq 0$ and $k \geq 0$. 
The main theorem in the present paper refines the result 
by describing its semi-orthogonal complement. 

\item There are several works categorifying some representation theoretic structures 
on the cohomology of the Hilbert schemes: 
Negut \cite{neg18} and Zhao \cite{zha20} studied a quantum troidal action 
on the derived categories of the Hilbert schemes; 
Porta--Sala \cite{ps21} constructed a categorified Hall product 
on the derived categories of Hilbert schemes 
(more generally on the moduli spaces of higher rank sheaves). 

\item Toda \cite{tod21c} obtained SODs for matrix factorization categories 
of moduli spaces on the resolved conifold $\Tot_{\bP^1}(\mcO(-1) \oplus \mcO(-1))$, 
which categorify Nagao-Nakajima's wall-crossing formula \cite{nn11}, 
and are similar to the result in this paper. 
\end{enumerate}

%=======================================================================
\subsection{Open questions}
\begin{enumerate}
\item As already observed in \cite{jia21}, 
the Quot formula have applications to various Brill-Noether type theory. 
It would be interesting to find more applications 
to the wall-crossing diagrams of the moduli spaces, 
e.g., for variations of Bridgeland stability conditions. 

\item It would be also interesting to generalize the result 
to the moduli space of higher rank sheaves on an arbitrary surface 
(not necessarily del Pezzo, K3, or abelian). 
In general, the moduli space is highly singular 
and hence its bounded derived category of coherent sheaves 
would not behave well. 
Alternatively, we may use the DT categories 
of the canonical bundle $\omega_S$ introduced by Toda \cite{tod21b}. 

\item By \cite{bkr01}, we have an equivalence 
$D^b(\Hilb^n(S)) \cong \Sym^n(D^b(S))$ 
for any surface $S$. 
Hence we can think of Theorem  \ref{thm:introhilb} (2) as 
the relation between the categories 
$ \Sym^n(D^b(\hatS))$ and  $\Sym^{n-i}(D^b(S))$. 
It would be interesting if one can generalize it in a purely categorical way. 
Namely, given a triangulated dg category $\hat{\mcC}$ with an SOD 
$\hat{\mcC}=\langle \mcC, E \rangle$, where $E \in \hat{\mcC}$ is an exceptional object, 
can one relate the category $\Sym^n(\hat{\mcC})$ with 
the categories $\Sym^{n-i}(\mcC)$?
\end{enumerate}

%=======================================================================
\subsection{Plan of the paper}
In Section \ref{sec:quot}, we recall the statement of the Quot formula. 
In Section \ref{sec:wc}, we recall the construction of wall-crossing diagrams 
due to Nakajima--Yoshioka. 
In Section \ref{sec:incidence}, we interpret moduli spaces on the blow-up 
as the Quot schemes. 
In Section \ref{sec:blformula}, we prove our main theorem.

%=======================================================================
\begin{ACK}
The author would like to thank Professors Arend Bayer and Yukinobu Toda, and Qingyuan Jiang 
for fruitful discussions. 
This work was supported by 
ERC Consolidator grant WallCrossAG, no.~819864. 

Finally the author would like to thank the referee 
for the careful reading of the previous version of this paper 
and giving him a lot of useful comments. 
The author would also like to thank Professor Wu-yen Chuang for pointing out an error in Theorem 5.4 in the previous version. 
\end{ACK}

%=======================================================================
\begin{NaC}
In this paper, we work over the complex number field $\bC$. 
For a scheme $X$, we denote by $\Coh(X)$ the abelian category of coherent sheaves on $X$, 
and by $D^b(X)$ the bounded derived category of coherent sheaves on $X$. 
\end{NaC}

%=======================================================================
\section{The Quot formula} \label{sec:quot}
In this section, we recall the {\it Quot formula}, 
which was conjectured by Jiang \cite{jia21} and proved by Toda \cite{tod21}. 

Let $X$ be a smooth quasi-projective variety, 
$\mcG$ a coherent sheaf on $X$ of rank $\delta \geq 0$ 
with homological dimension less than or equal to one. 
For a non-negative integer $d \geq 0$, we denote by 
\[
\Quot_{X, d}(\mcG) \to X
\]
the relative Quot scheme of rank $d$ locally free quotients of $\mcG$. 

Let $\mcK \coloneqq \mcE xt^1(\mcG, \mcO_X)$. 
Then the {\it expected dimensions} of 
the schemes $\Quot_{X, d}(\mcG), \Quot_{X, d}(\mcK)$ are 
\[
\dim X+\delta d-d^2, \quad \dim X-\delta d-d^2, 
\] 
respectively. 

We have the following Quot formula: 
\begin{thm}[{\cite[Theorem 1.1]{tod21}}] \label{thm:quotformula}
Let $d \geq 0$ be an integer. 
Suppose that the Quot schemes 
$\Quot_{X, d}(\mcG)$ and $\Quot_{X, d-i}(\mcK)$ are smooth and 
have expected dimensions for all $i=0, \cdots, \min\{d, \delta\}$. 
Then we have an SOD 
\[
D^b(\Quot_{X, d}(\mcG))=\left\langle
{\delta \choose i}\mbox{-copies of } D^b(\Quot_{X, d-i}(\mcK)) \colon 
i=0, \cdots, \min\{d, \delta\}
\right\rangle. 
\]
\end{thm}

\begin{rmk}
In \cite{tod21}, Toda proved the Quot formula without assuming 
the smoothness and the expected dimension condition of the Quot schemes. 
In that case, we need to encode the Quot schemes with certain quasi-smooth derived structures. 
See \cite[Remark 1.2]{tod21}. 
\end{rmk}

%=======================================================================
\section{Moduli spaces on the blow-up} \label{sec:wc}
Let $S$ be a smooth projective surface, 
$f \colon \hat{S} \to S$ be the blow-up of $S$ at a point $o \in S$. 
We denote by $C \subset \hatS$ the $f$-exceptional curve. 
Let $H$ be an ample divisor on $S$. 

\begin{defin}
Let $m \in \bZ_{\geq 0}$ be a non-negative integer. 
A coherent sheaf $E \in \Coh(\hatS)$ is {\it $m$-stable} if the following conditions hold: 
\begin{enumerate}
\item $\Hom(E(-mC), \mcO_C(-1))=0$, 
\item $\Hom(\mcO_C, E(-mC))=0$, 
\item $f_*E$ is $\mu_H$-stable. 
\end{enumerate}
\end{defin}

Take a cohomology class $w=(w_0, w_1, w_2) \in H^{2*}(S, \bQ)$. 
We work under the following assumption: 
\begin{assump} \label{ass:smooth}
The following three conditions hold: 
\begin{enumerate}
\item $w_0 >0$, 
\item $\gcd(w_0, H.w_1)=1$, 
\item one of the following conditions holds: 
\begin{enumerate}
\item $w_0=1, w_1=0$; 
\item $S$ is a del Pezzo, K3, or Abelian surface. 
\end{enumerate}
\end{enumerate}
\end{assump}

We denote by $M_S(w)$ (resp. $M_{\hatS}(v)$) 
the moduli space of $H$-Gieseker semistable sheaves on $S$ 
(resp. $(f^*H-\epsilon C)$-semistable sheaves on $\hatS$ with $0 < \epsilon \ll 1$) 
with Chern character $w$ (resp. $v \in H^{2*}(\hatS, \bQ)$), 
and we denote by $M^k(v)$ the moduli space of 
$k$-stable sheaves on $\hatS$ 
with the fixed Chern character $v \in H^{2*}(\hat{S}, \bQ)$. 

\begin{lem} \label{lem:sm}
Let $k, d \in \bZ_{\geq 0}$ be integers. 
Fix a class $w \in H^{2*}(S, \bQ)$ 
and put $v_d \coloneqq f^*w-d\ch(\mcO_C(-1)) \in H^{2*}(\hatS, \bQ)$. 
Under Assumption \ref{ass:smooth}, the moduli space $M^k(v_d)$ is 
either empty or smooth projective variety of dimension 
\begin{equation} \label{eq:modexpdim}
\Delta(w)-(w_0^2-1)\chi(\mcO_S)+h^1(\mcO_S)-d(w_0+d), 
\end{equation}
where $\Delta(w)\coloneqq w_1^2-2w_0w_2$ denotes the discriminant. 
\end{lem}
\begin{proof}

The projectivity follows from the Assumption \ref{ass:smooth} (2) 
(cf. \cite[Theorem 2.9]{ny11} and the sentence before Lemma 2.8 
in the same paper). 

For the smoothness, it is enough to show the vanishing 
\[
\Ext^2(E, E)_0 \cong \Hom(E, E \otimes \omega_{\hat{S}})^\vee_0=0 
\]
for all $E \in M^0(v_d)$ by \cite[Theorem 4.5.4]{hl97}. 
Here, for an integer $i$ and a line bundle $L$ on $\hatS$, 
$\Ext^i(E, E \otimes L)_0$ denotes the traceless part, i.e., 
\[
\Ext^i(E, E \otimes L)_0 \coloneqq
\ker\left(
\tr \colon \Ext^i(E, E \otimes L) \to H^i(L)
\right). 
\]

To prove this, first recall that there is a homomorphism 
$\iota \colon H^i(L) \to \Ext^i(E, E \otimes L)$ satisfying 
$\tr \circ \iota = \rk(E) \cdot \id \colon H^i(L) \to H^i(L)$ 
(see Lemma 10.1.3 and a comment before Definition 10.1.4 in \cite{hl97}). 
In particular, we have 
\[
\hom(E, E \otimes \omega_{\hatS})
=h^0(\omega_{\hatS})+\hom(E, E \otimes \omega_{\hatS})_0 
\]
since we assume $w_0=\rk(E) >0$. 
Moreover, by \cite[Lemma 3.6]{ny11}, we have an embedding 
\[
\Hom\left(E, E \otimes \omega_{\hatS} \right) \hookrightarrow 
\Hom\left((f_*E)^{\vee\vee},(f_*E)^{\vee\vee} \otimes \omega_{S} \right). 
\]
Combining these two facts, we have the following inequalities: 
\[
h^0(\omega_{\hatS}) \leq \hom(E, E \otimes \omega_{\hatS}) 
\leq \hom\left((f_*E)^{\vee\vee},(f_*E)^{\vee\vee} \otimes \omega_{S} \right).
\]

Now the problem is reduced to proving the following equalities: 
\[
\hom\left((f_*E)^{\vee\vee},(f_*E)^{\vee\vee} \otimes \omega_{S} \right)
=h^0(\omega_S)=h^0(\omega_{\hatS}). 
\]
The second equality holds since $\hatS$ and $S$ are birational. 
For the first equality, first note that 
$f_*(E(-kC))$ is $\mu_H$-stable by the definition of $k$-stability, 
and hence so is $(f_*(E(-kC)))^{\vee\vee}$. 
Moreover, since $f_*(E(-kC))$ and $f_*E$ are isomorphic away from 
the point $o \in S$, we have 
$(f_*(E(-kC)))^{\vee\vee} \cong (f_*E)^{\vee\vee}$. 
The desired equality then follows from Assumption \ref{ass:smooth} (3). 
Indeed, in the case (3-a), we have $(f_*E)^{\vee\vee} \cong \mcO_S$; 
in the case when $S$ is del Pezzo, we have 
$\hom\left((f_*E)^{\vee\vee},(f_*E)^{\vee\vee} \otimes \omega_{S} \right)
=0=h^0(\omega_S)$; 
in the case when $S$ is K3 or Abelian, we have 
$\hom\left((f_*E)^{\vee\vee},(f_*E)^{\vee\vee} \otimes \omega_{S} \right)
=1=h^0(\omega_S)$. 

The formula (\ref{eq:modexpdim}) of the dimension 
follows from the Riemann-Roch formula. 
\end{proof}

The following theorem is a part of the main results by Nakajima--Yoshioka, 
see \cite[equation (*) in page 48, Propositions 3.3 and 3.37]{ny11}: 
\begin{thm}[\cite{ny11}] \label{thm:NY}
Fix a class $w \in H^{2*}(S, \bQ)$ and put $v_0 \coloneqq f^*w$. 
There exists an integer $m_0 \geq 0$ and a sequence 
\begin{equation} \label{eq:NY}
M^0(v_0) \dashleftarrow M^1(v_0) \dashleftarrow \cdots 
\dashleftarrow M^{m-1}(v_0) \dashleftarrow M^m(v_0) \dashleftarrow \cdots 
\end{equation}
of birational maps satisfying 
$M^0(v_0) \cong M_S(w)$ and $M^m(v_0) \cong M_{\hatS}(v_0)$ 
for all $m \geq m_0$. 
\end{thm}

\begin{rmk}
By \cite[Theorem 1.3]{kos21}, the sequence (\ref{eq:NY}) are 
steps of an MMP for $M_{\hatS}(v_0)$. 
\end{rmk}

Note that $(-) \otimes \mcO(-kC)$ induces isomorphisms 
\begin{equation} \label{eq:tensor}
M^k(v) \cong M^0(v.e^{-kC}), \quad 
M^{k+1}(v) \cong M^1(v.e^{-kC}) 
\end{equation}
for each $k \geq 0$ and any class $v \in H^{2*}(\hatS, \bQ)$.

%=======================================================================
\section{Wall-crossing via Quot schemes} \label{sec:incidence}
In this section, we recall 
the interpretation of the moduli spaces $M^k(v)$ 
as the Quot schemes following Nakojima--Yoshioka. 
Throughout this section, we fix a class 
$w=(w_0, w_1, w_2) \in H^{2*}(S, \bQ)$ satisfying 
$w_0>0$ and $\gcd(w_0, H.w_1)=1$. 

\begin{thm}[{\cite[Theorem 4.1]{ny11}}] \label{thm:incidence}
Let $\mcF \in \Coh(M_S(w) \times S)$ be the universal sheaf. 
We put $\mcG \coloneqq \mcF|_{M_S(w) \times \{o\}}$ and 
$\mcK \coloneqq \mcE xt^1(\mcG, \mcO_{M_S(w)})$. 
Let $d \in \bZ_{>0}$ be a positive integer. 
The following statements hold: 
\begin{enumerate}
\item The sets of closed points of the Quot schemes 
$\Quot_{M_S(w), d}(\mcK)$ and $\Quot_{M_S(w), d+w_0}(\mcG)$ 
are given as follows: 
\begin{equation} \label{eq:incidence}
\begin{aligned}
&\Quot_{M_S(w), d}(\mcK)(\Spec\bC) = 
\left\{
(F, U) \colon 
\begin{aligned}
&F \in M_S(w), \\
&U \subset \Ext^1(\mcO_o, F), \dim U=d
\end{aligned}
\right\}, \\
&\Quot_{M_S(w), d+w_0}(\mcG)(\Spec\bC) = 
\left\{
(F, V) \colon 
\begin{aligned}
&F \in M_S(w), \\
&V \subset \Hom(F, \mcO_o), \dim V=d+w_0
\end{aligned}
\right\}. \\
\end{aligned}
\end{equation}

\item Let us put $v_d \coloneqq f^*w-d\ch(\mcO_C(-1))$. 
Then we have isomorphisms 
\begin{equation} \label{eq:isomquot}
\begin{aligned}
M^0(v_d) \cong \Quot_{M_S(w), d}(\mcK), \quad 
M^1(v_d) \cong \Quot_{M_S(w), d+w_0}(\mcG). 
\end{aligned}
\end{equation}
\end{enumerate}
\end{thm}
\begin{proof}
(1) We only prove the assertion for $\Quot_{M_S(w), d+w_0}(\mcG)$. 
By the definition of the Quot scheme, 
the fiber of $\Quot_{M_S(w), d+w_0}(\mcG) \to M_S(w)$ 
at a point $F \in M_S(w)$ parametrizes quotients of 
the vector space $F|_{\{o\}}$. 
Since we have isomorphisms  
\[
F|_{\{o\}} \cong \Hom_{\{o\}}(F|_{\{o\}}, \mcO_o)^\vee\cong 
\Hom(F, \mcO_o)^\vee, 
\]
giving a quotient $F|_{\{o\}} \twoheadrightarrow V^\vee$ 
is equivalent to giving a subspace 
$V \subset \Hom(F, \mcO_o)$. 
This proves the first assertion. 

(2) The second assertion is proved in \cite[Theorem 4.1]{ny11}. 
\end{proof}

\begin{lem} \label{lem:twoterm}
Suppose that Assumption \ref{ass:smooth} holds. 
Let $\mcF \in \Coh(M_S(w) \times S)$ be the universal sheaf. 
Then there exists an exact sequence 
\[
0 \to \mcV_0 \to \mcV_1 \to \mcF|_{M_S(w) \times \{o\}} \to 0
\]
for some vector bundles $\mcV_0, \mcV_1$. 
\end{lem}
\begin{proof}
Since $H^i(S, F \otimes \mcO_o)=0$ 
for all $F \in M_S(w)$ and $i \neq 0, 1$, 
the sheaf $\mcF|_{M_S(w) \times \{o\}}$ 
is represented by a two term complex of vector bundles: 
\[
\mcF|_{M_S(w) \times \{o\}} \cong (\mcV_0 \xrightarrow{\phi} \mcV_1). 
\]
Indeed, since $M_S(w)$ is smooth, we have an isomorphsim 
$\mcV_\bullet \cong \mcF|_{M_S(w) \times \{o\}}$ 
in the derived category $D^b(M_S(w))$ 
for some perfect complex $\mcV_\bullet$. 
Now the above vanishing implies that 
$\mcH^i(\mcV_\bullet)=0$ for all $i \neq 0, 1$ 
by the base change and Nakayama's lemma. 
Hence the complex $\mcV_\bullet$ is isomorphic to 
the two term complex of vector bundles. 

To see that $\phi \colon \mcV_0 \to \mcV_1$ is injective, 
it is enough to show that the locus 
\[
M^{\geq 1}_S(w) \coloneqq \left\{ 
F \in M_S(w) \colon h^1(F \otimes \mcO_o) \geq 1 
\right\} \subset M_S(w)
\]
is a proper subset. 
Since we have $h^1(F \otimes \mcO_o)=\ext^1(F, \mcO_o)$, 
the locus $M^{\geq 1}_S(w)$ coincides 
with the image of the natural morphism 
\[
M^0(f^*w-\ch(\mcO_C(-1))) \cong \Quot_{M_S(w), 1}(\mcK) \to M_S(w), 
\]
where the first isomorphism follows from Theorem \ref{thm:incidence}. 
If $\Quot_{M_S(w), d}(\mcK) = \emptyset$, 
then we also have $M^{\geq 1}_S(w) = \emptyset$ 
and the result holds. 
Otherwise, by Lemma \ref{lem:sm}, we have the following inequality as required: 
\begin{align*}
\dim M^{\geq 1}_S(w) 
&\leq \dim M^0(f^*w-\ch(\mcO_C(-1))) \\
&=\dim M_S(w)-(w_0+1) < \dim M_S(w). 
\end{align*}
\end{proof}

By Theorem \ref{thm:incidence} and Lemma \ref{lem:twoterm}, 
we obtain the following: 
\begin{cor} \label{cor:quot}
Suppose that Assumption \ref{ass:smooth} holds. 
Keeping the notations in Theorem \ref{thm:incidence}, 
the following statements hold: 
\begin{enumerate}
\item The sheaf 
$\mcG \coloneqq \mcF|_{M_S(w) \times \{o\}}$ 
has homological dimension one. 

\item The Quot schemes (\ref{eq:isomquot}) are smooth and 
have expected dimensions. 
\end{enumerate}
\end{cor}
\begin{proof}
The first statement follows from Lemma \ref{lem:twoterm}. 
The second statement follows from Lemma \ref{lem:sm}. 
\end{proof}

%=======================================================================
\section{Categorical blow-up formula} \label{sec:blformula}
In this section, we prove a categorification of the blow-up formula. 

\subsection{Main theorem}
First we relate the derived categories of 
the moduli spaces of $1$-stable sheaves with 
that of the moduli spaces of $0$-stable sheaves: 

\begin{prop} \label{prop:DbM1}
Fix a class $w \in H^{2*}(S, \bQ)$ and 
suppose that Assumption \ref{ass:smooth} holds. 
For any non-negative integer $d \in \bZ_{\geq 0}$, 
we have an SOD
\begin{equation} \label{eq:DbM1}
\begin{aligned}
&\quad D^b(M^{1}(f^*w-d\ch(\mcO_C(-1)))) \\
&=\left\langle 
{w_0 \choose i}\mbox{-copies of } 
D^b(M^{0}(f^*w-(w_0+d-i)\ch(\mcO_C(-1)))) : 
0 \leq i \leq w_0
\right\rangle. 
\end{aligned}
\end{equation}
\end{prop}
\begin{proof}
Recall that we have 
\[
M^{1}(f^*w-d\ch(\mcO_C(-1)))
\cong \Quot_{M_S(w), w_0+d}(\mcG), 
\]
by Theorem \ref{thm:incidence} (2). 
Note that $\mcG=\mcF|_{M_S(w) \times \{o\}}$ 
has rank $w_0$. 
By Corollary \ref{cor:quot}, we can apply 
the Quot formula (Theorem \ref{thm:quotformula}) 
and get 
\begin{align*}
&\quad D^b(M^{1}(f^*w-d\ch(\mcO_C(-1)))) \\
&=\left\langle
{w_0 \choose i}\mbox{-copies of } D^b(\Quot_{M_S(w), w_0+d-i}(\mcK)) 
\colon 0 \leq i \leq w_0
\right\rangle. 
\end{align*} 
By using Theorem \ref{thm:incidence} (2) again, we have 
\[
\Quot_{M_S(w), w_0+d-i}(\mcK) \cong M^{0}(f^*w-(w_0+d-i)\ch(\mcO_C(-1)))
\]
as required. 
\end{proof}

We will use the above proposition recursively 
to relate the moduli spaces of $d$-stable sheves on the blow-up $\hatS$ 
with the moduli spaces of Gieseker stable sheaves 
on the original surface $S$. 
Before stating the main result, we fix some notations. 

For integers $r, d, j \in \bZ$ with $r>0$ and $d \geq 0$, 
we define 
\begin{equation} \label{eq:ThetaA}
\begin{aligned}
&\Theta_{r, d}(j) \coloneqq \left\{
\vec{k}=(k_1, \cdots, k_{l}) \in \bZ^{l} \colon 
\begin{aligned}
&l \in \bZ_{\geq 0}, 
0 \leq k_i \leq r, \\
&\sum_i k_i = rd, 
\sum_i ik_i = j+rd(d+1)/2
\end{aligned}
\right\}, \\
&\widetilde{A}_{r, d}(j) \coloneqq \sum_{\vec{k} \in \Theta_{r, d}(j)} 
\prod_{i=1}^l
{r \choose k_i}. 
\end{aligned}
\end{equation}
See the next subsection for a combinatorial meaning of the number 
$\widetilde{A}_{r, d}(j)$. 

\begin{thm} \label{thm:sodgeneral}
Fix a class $w \in H^{2*}(S, \bQ)$ and 
suppose that Assumption \ref{ass:smooth} holds. 
Let $d \in \bZ_{\geq 0}$ be a non-negative integer. 
Then we have an SOD
\begin{align*}
D^b(M^{d+1}(f^*w))=\left\langle
\widetilde{A}_{w_0, d+1}(j) \mbox{-copies of } 
D^b(M_S(f^*w+(0, 0, j))) \colon 
j \in \bZ_{\geq 0}
\right\rangle. 
\end{align*}
\end{thm}
\begin{proof}
We first apply Proposition \ref{prop:DbM1} to the scheme 
\[
M^{d+1}(f^*w) \cong M^1(f^*w'-w_0d\ch(\mcO_C(-1))), 
\]
where we put 
\begin{equation} \label{eq:w'}
w' \coloneqq (w_0, w_1, w_2-w_0d(d+1)/2) 
\end{equation}
(see the isomorphism (\ref{eq:tensor})). 
Then the SOD (\ref{eq:DbM1}) consists of the following components: 
\begin{equation} \label{eq:1ststep}
{w_0 \choose k_1} \mbox{-copies of }
D^b(M^0(f^*w'-(w_0(d+1)-k_1))\ch(\mcO(-1))), \quad 
0 \leq k_1 \leq w_0. 
\end{equation}
By (\ref{eq:tensor}), we have an isomorphism 
\begin{align*}
&\quad M^0(f^*w'-(w_0(d+1)-k_1)\ch(\mcO(-1))) \\
& \cong 
M^1(f^*w'-(w_0d-k_1)\ch(\mcO(-1))+(0, 0, w_0(d+1)-k_1)). 
\end{align*}
Hence we can keep applying Proposition \ref{prop:DbM1} 
to each component (\ref{eq:1ststep}). 
Note that if $k_1=0$, then we have 
\[
\dim M^0(f^*w'-w_0(d+1)\ch(\mcO(-1)))
<\dim M^{d+1}(f^*w)
\]
by (\ref{eq:modexpdim}). 
Hence by induction on $d$ and the dimension of the moduli space, 
we obtain an SOD consisting of the categories 
\[
D^b(M^0(f^*w+(0, 0, j))), \quad j \geq 0. 
\]
Explicitly, by applying Proposition \ref{prop:DbM1} $(l+1)$-times 
(where $l \in \bZ_{\geq 0}$), 
we obtain 
\[
\prod_{i=1}^{l+1} {w_0 \choose k_i} \mbox{-copies of } 
D^b(M^0(f^*w'-(w_0(d+1)-\sum_ik_i)\ch(\mcO_C(-1))+(0, 0, s(\vec{k}))), 
\]
as semi-orthogonal summands of $D^b(M^{d+1}(f^*w))$, 
where $0 \leq k_i \leq w_0$ and we put 
\begin{equation} \label{eq:defs}
s(\vec{k}) \coloneqq (l+1)w_0(d+1)-(l+1)k_1-lk_2-\cdots-2k_l-k_{l+1}. 
\end{equation}
We continue it until we have 
\begin{equation} \label{eq:sumki}
w_0(d+1)-\sum_ik_i=0. 
\end{equation}
By substituting (\ref{eq:sumki}) to (\ref{eq:defs}), we obtain 
\begin{equation} \label{eq:defs2}
s(\vec{k})=k_2+2k_3+\cdots+k_{l+1}. 
\end{equation}

Let $j \in \bZ_{\geq 0}$ be a non-negative integer. 
Summarizing the above arguments 
and recalling the definition of $w'$ from (\ref{eq:w'}), 
the number of the category 
$D^b(M^0(f^*w+(0, 0, j)))$ 
in the SOD of $D^b(M^{d+1}(f^*w))$ is 
\begin{equation} \label{eq:numbercomps}
\sum_{\vec{k}}\prod_i {w_0 \choose k_i}, 
\end{equation}
where the summation runs over tuples 
$\vec{k}=(k_1, \cdots, k_{l+1})$ 
of non-negative integers satisfying the conditions 
\[
\sum_i k_i=w_0(d+1), \quad 
s(\vec{k})=j+w_0d(d+1)/2.  
\]
Adding $w_0(d+1)=\sum_i k_i$ 
to both sides of the second equation, 
and using (\ref{eq:defs2}), we get 
\[
\sum_i ik_i=j+w_0(d+1)(d+2)/2. 
\]
In other words, the vectors $\vec{k}$ in (\ref{eq:numbercomps}) are 
exactly elements of $\Theta_{w_0, d+1}(j)$, 
and the number (\ref{eq:numbercomps}) is equal to 
$\widetilde{A}_{w_0, d+1}(j)$ 
(see (\ref{eq:ThetaA}) for the definitions of 
$\Theta_{w_0, d+1}(j)$ and $\widetilde{A}_{w_0, d+1}(j)$). 

Finally, we have an isomorphism 
$M^0(f^*w+(0, 0, j)) \cong M_S(w+(0, 0, j))$ 
for each $j \in \bZ_{\geq 0}$ 
by Theorem \ref{thm:NY} 
and obtain the desired SOD. 
\end{proof}

%======================================================================
\subsection{Some Combinatorics}
In this subsection, 
we explain the combinatorial meaning of the numbers 
$\widetilde{A}_{r, d}(j)$ 
defined in (\ref{eq:ThetaA}), 
and compare our SOD in Theorem \ref{thm:sodgeneral} with 
the numerical formula due to Nakajima--Yoshioka \cite{ny11}. 

Fix integers $r>0, d \geq 0, j \geq 0$. 
We use the following notations: 
\begin{itemize}
\item For a Young diagram $Y$, $|Y|$ denotes the number of boxes in $Y$, 
and $c(Y)$ denotes the number of columns in $Y$. 

\item For an $r$-tuple $\vec{Y}=(Y_1, \cdots, Y_r)$ of Young diagrams, 
we put $|\vec{Y}| \coloneqq \sum_\alpha|Y_\alpha|$. 

\item For an $r$-tuple $\vec{m}=(m_1, \cdots, m_r)$ of integers, 
we put 
\begin{equation} \label{eq:definner}
(\vec{m}, \vec{m}) \coloneqq 
\frac{\sum_{\alpha, \beta}(m_\alpha-m_\beta)^2}{2r}. 
\end{equation}
\end{itemize}

We then define a number $A_{r, d}(j)$ as follows: 
\begin{align*}
&A_{r, d}(j) \coloneqq \#\left\{
(m_\alpha, Y_\alpha)_{\alpha=1}^r \colon 
\begin{aligned}
&m_\alpha \in \bZ_{\geq 0}, \quad \sum_\alpha m_\alpha=rd, \\
&Y_\alpha \mbox{ is a Young diagram with } 
c(Y_\alpha) \leq m_\alpha, \\ 
&|\vec{Y}|+(\vec{m}, \vec{m})/2=j
\end{aligned}
\right\}. 
\end{align*}
Note that the transformation 
$(m_\alpha)_\alpha \mapsto (m_\alpha-d)_\alpha$ 
does not change the value $(\vec{m}, \vec{m})$. 
Hence, the generating series of $A_{r, d}(j)$ 
has the following expression: 
\[
\sum_{j \geq 0} A_{r, d}(j) q^j
=\sum_{\substack{m_\alpha \geq -d \\ m_1+\cdots m_{r}=0}}
\left(\prod_\alpha \prod_{k=1}^{m_\alpha+d} \frac{1}{1-q^k}
\right)
\cdot q^{\frac{(\vec{m}, \Vec{m})}{2}}. 
\]
We also define a number $A_{r, +\infty}(j)$ by the following formula:
\[
\sum_{j \geq 0} A_{r, +\infty}(j) q^j
=\left(\prod_{k=1}^\infty \frac{1}{1-q^k} \right)^r
\cdot \sum_{\substack{m_\alpha \in \bZ \\ m_1+\cdots m_{r}=0}} 
q^{\frac{(\vec{m}, \Vec{m})}{2}}. 
\]
For $r=1$, $A_{1, +\infty}(j)$ agrees with the partition function $p(j)$. 

The following lemma would be well-known for experts, 
but we include the proof here for the readers' convenience: 
\begin{lem} \label{lem:combi}
Let $r > 0, d \geq 0, j \geq 0$ be integers. 
We have an equality 
$\widetilde{A}_{r, d}(j)={A}_{r, d}(j)$. 
\end{lem}
\begin{proof}
Let us first consider the case $r=1$. 
In this case, 
the set $\Theta_{1, d}(j)$ defined in (\ref{eq:ThetaA}) is bijective to 
the set of strictly increasing sequences $0 < i_1 < \cdots < i_d$ of 
positive integers satisfying 
\[
\sum_t i_t=j+\frac{d(d+1)}{2}. 
\]
To each such sequence $\vec{i}$, we associate a Young diagram $Y_{\vec{i}}$ 
whose number of boxes in the $t$-th column equals to $j_t-t$. 
This gives a bijection between the set $\Theta_{1, d}(j)$ with 
the set of Young diagrams $Y$ satisfying 
$|Y|=j$ and $c(Y) \leq d$. 
Hence the desired equality 
$\widetilde{A}_{1, d}(j)={A}_{1, d}(j)$ holds. 

Let us now consider the case $r \geq 2$. 
For each element $\vec{k} \in \Theta_{r, d}(j)$ 
and $i=1, \cdots, l$, 
there are ${r \choose k_i}$ ways to write $k_i$ as 
\[
k_i=\epsilon^{(i)}_1+\cdots+\epsilon^{(i)}_r, 
\quad 0 \leq \epsilon^{(i)}_\alpha \leq 1. 
\]
Given such expressions for $i=1, \cdots, l$, 
we can associate elements 
\[
\vec{\epsilon}_\alpha=
(\epsilon^{(i)}_\alpha)_{i=1}^l \in \Theta_{1, m_\alpha}(j_\alpha), 
\quad \alpha=1, \cdots, r, 
\]
where we put $m_\alpha \coloneqq \sum_i\epsilon^{(i)}_\alpha$ and 
$j_\alpha \coloneqq \sum_i i \epsilon^{(i)}_\alpha-m_\alpha(m_\alpha+1)/2$. 
This gives us an equality 
\begin{equation} \label{eq:reduction}
\widetilde{A}_{r, d}(j)=\#\left\{
(\vec{\epsilon}_\alpha)_{\alpha=1}^r 
\in \prod_{\alpha=1}^r \Theta_{1, m_\alpha}(j_\alpha) 
\colon \begin{aligned}
&m_\alpha \geq 0, j_\alpha \geq 0, \sum_\alpha m_\alpha=rd, \\
&\sum_\alpha \big(m_\alpha(m_\alpha+1)/2+j_\alpha \big)=j+rd(d+1)/2
\end{aligned}
\right\}. 
\end{equation}
By the rank one case treated above, 
the right hand side of (\ref{eq:reduction}) equals to 
\begin{equation*}
\#\left\{
(m_\alpha, Y_\alpha)_{\alpha=1}^r \colon 
\begin{aligned}
&m_\alpha \geq 0, \sum_\alpha m_\alpha=rd, \\
&Y_\alpha \mbox{ is a Young diagram with } c(Y_\alpha) \leq m_\alpha, \\
&\sum_\alpha m_\alpha(m_\alpha+1)/2+|\vec{Y}|=j+rd(d+1)/2
\end{aligned}
\right\}. 
\end{equation*}
Now it is enough to show that the equations 
\begin{equation} \label{eq:inner}
|\vec{Y}|+(\vec{m}, \vec{m})/2=j
\end{equation}
and 
\begin{equation} \label{eq:sums}
\sum_\alpha m_\alpha(m_\alpha+1)/2+|\vec{Y}|=j+rd(d+1)/2
\end{equation}
are equivalent. 
By using $\sum_\alpha m_\alpha=rd$, (\ref{eq:sums}) is equivalent to 
\begin{equation}\label{eq:sums2}
\sum_\alpha m_\alpha^2/2+|\vec{Y}|=j+d^2/2. 
\end{equation}
On the other hand, recall from (\ref{eq:definner}) that 
we defined $(\vec{m}, \vec{m})$ as 
$(\vec{m}, \vec{m})=\sum_{\alpha, \beta}(m_\alpha-m_\beta)^2/(2r)$. 
We have 
\begin{align*}
\sum_{\alpha, \beta}(m_\alpha-m_\beta)^2
&=2(r-1)\sum_\alpha m_\alpha^2-\sum_{\alpha \neq \beta}m_\alpha m_\beta \\
&=2(r-1)\sum_\alpha m_\alpha^2
-2\big(\sum_\alpha m_\alpha \big)\big(\sum_\beta m_\beta \big)
+2\sum_\alpha m_\alpha^2 \\
&=2r\sum_\alpha m_\alpha^2-2r^2d^2, 
\end{align*}
where the third equality follows from $\sum_\alpha m_\alpha=rd$. 
It follows that (\ref{eq:inner}) is equivalent to (\ref{eq:sums2}) 
as desired. 
\end{proof}

Let us end this paper by rephrasing Theorem \ref{thm:sodgeneral} 
based on Lemma \ref{lem:combi}: 

\begin{thm} \label{thm:sodfinal}
Fix a class $w \in H^{2*}(S, \bQ)$ and 
suppose that Assumption \ref{ass:smooth} holds. 
The following statements hold: 
\begin{enumerate}
\item Let $d \in \bZ_{\geq 0}$ be a non-negative integer. 
Then we have an SOD
\begin{align*}
D^b(M^{d}(f^*w))=\left\langle
A_{w_0, d}(j) \mbox{-copies of } 
D^b(M_S(f^*w+(0, 0, j))) \colon 
j \in \bZ_{\geq 0}
\right\rangle, 
\end{align*}
where the numbers $A_{w_0, d}(j)$ are defined by the following formula: 
\[
\sum_{j \geq 0} A_{w_0, d}(j) q^j
=\sum_{\substack{m_\alpha \geq -d \\ m_1+\cdots m_{w_0}=0}}
\left(\prod_\alpha \prod_{k=1}^{m_\alpha+d} \frac{1}{1-q^k}
\right)
\cdot q^{\frac{(\vec{m}, \Vec{m})}{2}}. 
\]

\item We have an SOD 
\begin{align*}
D^b(M_{\hatS}(f^*w))=\left\langle
A_{w_0, +\infty}(j) \mbox{-copies of } 
D^b(M_S(f^*w+(0, 0, j))) \colon 
j \in \bZ_{\geq 0}
\right\rangle, 
\end{align*}
where the numbers $A_{w_0, +\infty}(j)$ are defined by the following formula: 
\[
\sum_{j \geq 0} A_{w_0, +\infty}(j) q^j
=\left(\prod_{k=1}^\infty \frac{1}{1-q^k} \right)^{w_0}
\cdot \sum_{\substack{m_\alpha \in \bZ \\ m_1+\cdots m_{w_0}=0}} 
q^{\frac{(\vec{m}, \Vec{m})}{2}}. 
\]

\item In particular, putting $w=(1, 0, -n)$, $n \in \bZ_{>0}$, 
we have an SOD 
\[
D^b(\Hilb^n(\hatS))=\left\langle
p(j)\mbox{-copies of } D^b(\Hilb^{n-j}(S)) \colon j=0, \cdots n
\right\rangle, 
\]
where $p(j)$ is the partition function. 
\end{enumerate}
\end{thm}

Note that the combinatorial coefficients $A_{w_0, d}(j)$ are 
exactly the ones appearing in the corresponding formula for 
the Euler characteristics in \cite[Corollary 5.7]{ny11} with $t=-1$ 
(see also \cite[Theorem 3.21]{ny04}).

\end{document}